\documentclass[journal, web]{ieeecolor}

\def\logoname{LOGO-white-web}
\definecolor{subsectioncolor}{rgb}{0.067,0.627,0.859}
\usepackage{cite}
\usepackage{amsmath,amssymb,amsfonts}
\usepackage{algorithmic}
\usepackage{graphicx} 
\usepackage{textcomp}
\def\BibTeX{{\rm B\kern-.05em{\sc i\kern-.025em b}\kern-.08em
    T\kern-.1667em\lower.7ex\hbox{E}\kern-.125emX}}
    
\usepackage{mathrsfs}
\usepackage{mathtools}
\usepackage{tikz}
\usepackage{pgfplots}
\DeclareUnicodeCharacter{2212}{−}
\usepgfplotslibrary{groupplots, dateplot}
\usetikzlibrary{patterns, shapes.arrows}
\pgfplotsset{compat=newest}
\newlength\axisheight
\newlength\axiswidth
\usepackage{xcolor}

\newtheorem{theorem}{Theorem}
\newtheorem{assumption}{Assumption}

\newtheorem{remark}{Remark}
\newtheorem{lemma}{Lemma}

\newtheorem{proposition}{Proposition}
\newtheorem{example}{Example}

\DeclareMathOperator*{\argmin}{arg\,min}

\begin{document}
\title{Transient Performance of MPC for Tracking}
\author{Matthias K{\"o}hler, Lisa Kr{\"u}gel, Lars Gr{\"u}ne, Matthias A. M{\"u}ller, and Frank Allg{\"o}wer
\thanks{The authors are thankful that this work was funded by the Deutsche Forschungsgemeinschaft (DFG, German Research Foundation) -- AL 316/11-2 - 244600449; GR 1569/13-2 - 244602989; and under Germany's Excellence Strategy -- EXC 2075 -- 390740016.}
\thanks{Matthias K{\"o}hler and Frank Allg{\"o}wer are with the University of Stuttgart, Institute for Systems Theory and Automatic Control, Stuttgart, Germany
(e-mail: \{matthias.koehler, allgower\}@ist.uni-stuttgart.de)}
\thanks{Lisa Kr{\"u}gel and Lars Gr{\"u}ne are with the Universit{\"a}t Bayreuth, Mathematical Institute, Chair of Applied Mathematics, Bayreuth, Germany
(e-mail: \{lisa.kruegel, lars.gruene\}@uni-bayreuth.de)}
\thanks{Matthias A. M{\"u}ller is with the Leibniz University Hannover, Institute of Automatic Control, Hanover, Germany
(e-mail: mueller@irt.uni-hannover.de)}}


\pubid{\begin{minipage}{\textwidth}\ \\[14pt] \\ \\
\scriptsize
\copyright 2023 IEEE. Personal use of this material is permitted. Permission 
from IEEE must be obtained for all other uses, in any current or future 
media, including reprinting/republishing this material for advertising or 
promotional purposes, creating new collective works, for resale or 
redistribution to servers or lists, or reuse of any copyrighted 
component of this work in other works. See https://doi.org/10.1109/LCSYS.2023.3287798\\
M. K\"ohler, L. Kr\"ugel, L. Gr\"une, M. A. M\"uller, and F. Allg\"ower, "Transient Performance of MPC for Tracking," \emph{IEEE Control Syst. Lett.}, vol. 7, pp. 2545-2550, 2023.
\end{minipage}}

\maketitle

\begin{abstract}
    We analyse the closed-loop performance of a model predictive control (MPC) for tracking formulation with artificial references.
    It has been shown that such a scheme guarantees closed-loop stability and recursive feasibility for any externally supplied reference, even if it is unreachable or time-varying.
    The basic idea is to consider an artificial reference as an additional decision variable and to formulate generalised terminal ingredients with respect to it.
    In addition, its offset is penalised in the MPC optimisation problem, leading to closed-loop convergence to the best reachable reference. 
    In this paper, we provide a transient performance bound on the closed loop using MPC for tracking.
    We employ mild assumptions on the offset cost and scale it with the prediction horizon.
    In this case, an increasing horizon in MPC for tracking recovers the infinite horizon optimal solution. 
\end{abstract}

\begin{IEEEkeywords}
Model predictive control, nonlinear systems, set-point tracking, transient performance
\end{IEEEkeywords}

\section{Introduction}
Nonlinear model predictive control (MPC) is a highly successful control strategy for nonlinear systems subject to constraints, guaranteeing stability and constraint satisfaction while optimising a performance objective.
It has a wide span of application, e.g. in robotics, energy, or process industry.
For a general review of MPC, see, e.g.~\cite{Rawlings2020,Gruene2017}.
A common objective in MPC is stabilisation of an external, possibly changing, reference, which may not correspond to an equilibrium of the plant and for which offline designed parts of MPC, e.g. fixed terminal constraints, are not suitable.

Recently, \emph{MPC for tracking}~\cite{Limon2008} has been introduced as a tracking MPC scheme that uses an artificial reference as an additional decision variable in the MPC's optimisation problem and penalises its offset to the current external reference.
In comparison to standard MPC, it retains feasibility despite changes in the external reference, it possibly allows for shorter prediction horizons, and it guarantees convergence to the best reachable equilibrium with respect to the external reference, which may not be an equilibrium of the plant or unreachable. 
MPC for tracking has been extended to, e.g. nonlinear systems with terminal ingredients~\cite{Limon2018} and without~\cite{Soloperto2022}, as well as to periodic references for linear~\cite{Limon2016} and nonlinear systems~\cite{Koehler2020b}.
Moreover, it has been extended to and applied in distributed MPC, see, e.g.~\cite{Conte2013a,Aboudonia2022a,Koehler2022b}.

Regarding performance, it has been shown for linear systems that MPC for tracking is optimal in a neighbourhood of the external reference~\cite{Ferramosca2009, Ferramosca2011a} for a suitably chosen offset cost.
Apart from this, the degradation of the closed-loop performance when using MPC for tracking, a possible disadvantage, has not been examined.
In this paper, we present a standard MPC for tracking scheme for nonlinear systems and derive bounds on its transient closed-loop performance, which depend on the choice of the offset cost.
In particular, for a suitably chosen offset cost, we show that the closed loop recovers an infinite horizon optimal solution if the horizon tends to infinity.
This is, in general, only the case if the offset cost is scaled with the horizon, which we demonstrate with a counterexample.
We illustrate our results in a simulation of a continuous stirred-tank reactor.

\subsection{Notation}
The natural numbers including 0 are denoted by $\mathbb{N}_0$.
The interior of a set $\mathcal{A}$ is $\mathrm{int}\, \mathcal{A}$.
The set of non-negative reals is $\mathbb{R}_{\ge 0}$.
The set of integers from $a$ to $b$, $a \le b$, is $\mathbb{I}_{a:b}$.
Let $\mathcal{V}$ be a normed space and $\mathcal{V}' \subseteq \mathcal{V}$ be a closed subset.
Then, for $v \in \mathcal{V}$, $\vert v \vert_{\mathcal{V}'} = \min_{v'\in\mathcal{V}'} \Vert v - v' \Vert_{2}$ defines the Euclidean distance of $v$ to $\mathcal{V}'$.
If $\mathcal{V}' = \{v'\}$ for some $v' \in  \mathcal{V}$, we write $\vert v \vert_{v'}$.
The largest eigenvalue of a matrix $A = A^\top$ is denoted by $\mu_{\mathrm{max}}(A)$.
If A is also positive (semi-)definite, we write $\Vert x \Vert_A^2 = x^\top A x$.
We define $\mathcal{B}_{c}(\tilde{x}) = \{ x \in \mathbb{R}^n \mid \Vert x - \tilde{x} \Vert_{2} \le c \}$.
We make use of comparison functions, see, e.g.~\cite{Kellett2014}.

\section{Problem formulation}
Consider a nonlinear, time-invariant discrete-time system
$
    x(t+1) = f(x(t), u(t)) 
$
with state $x(t) \in X \subset \mathbb{R}^n$, input $u(t) \in U \subset \, \mathbb{R}^m$, where $X$ is bounded and $U$ is compact, and continuous dynamics $f: \mathbb{R}^n \times \mathbb{R}^m \to \mathbb{R}^n$.
It is subject to pointwise state and input constraints $(x(t), u(t)) \in \mathcal{Z}$ with compact $\mathcal{Z} \subseteq X \times U$.

Given an initial state $x \in X$ and an input sequence $u = (u(0), u(1), \dots) \in U^K$, the solution starting at $x$ is denoted by $x_u(k, x)$, $k \in \mathbb{I}_{0:K-1}$.
We abbreviate $x_u(k) = x_u(k, x)$ when the initial state $x$ is obvious.
An input sequence $u \in U^K$ is called \emph{admissible}, if $(x_u(k, x), u(k)) \in \mathcal{Z}$ for all $k \in \mathbb{I}_{0:K-1}$ and $x_u(K,x) \in X$.
The set of admissible input sequences $u$ of length $K \in \mathbb{N}\cup\{\infty\}$ starting at $x$ and such that $x_u(K, x) \in \mathbb{X}$ (or $\lim_{K\to\infty} \vert x_u(K, x) \vert_\mathbb{X} = 0$ if $K=\infty$) with some closed set $\mathbb{X}$ is denoted by $\mathbb{U}_{\mathbb{X}}^K(x)$. 

Let $\mathcal{Z}_{\mathrm{r}} \subseteq \mathrm{int}\,\mathcal{Z}$ be closed and define the set of (admissible) references
$
    \mathcal{R} = \{ (x,u) \in \mathcal{Z}_{\mathrm{r}} \mid x = f(x,u) \}.
$
Given $r = (x_r, u_r), \hat{r} = (x_{\hat{r}}, u_{\hat{r}}) \in \mathcal{R}$, define $\vert r \vert_{\hat{r}} = \sqrt{\vert x_r \vert_{x_{\hat{r}}}^2 + \vert u_r \vert_{u_{\hat{r}}}^2}$.

The control goal is to steer the system towards an external reference $r_{\mathrm{e}}$ while satisfying the constraints and minimising a stage cost $\ell: X \times U \times \mathcal{R} \to \mathbb{R}_{\ge 0}$.
If the external reference is not reachable or not an equilibrium of the system, i.e. $r_{\mathrm{e}} \notin \mathcal{R}$, the control goal is to steer the system as close as admissible towards $r_{\mathrm{e}}$.
For this purpose, an offset cost $T$ is introduced, which attains its minimum at $r_{\mathrm{e}}$.
Then, the control goal is to steer the system to an admissible reference that minimises $T$ over $\mathcal{R}$.
For simplicity, we assume that this best reachable reference
$r_{\mathrm{d}} = (x_{\mathrm{d}}, u_{\mathrm{d}}) = \argmin_{r \in \mathcal{R}} T(r)$
is unique and that $T$ is a good indicator of $\vert r \vert_{r_{\mathrm{d}}}$.
\begin{assumption}\label{asm:offset_cost_indication}
    There exist $\alpha_{\mathrm{lo}}^T, \alpha_{\mathrm{up}}^T \in \mathcal{K}_{\infty}$ such that for all $r\in\mathcal{R}$ (not necessarily for all equilibria, e.g. $r_{\mathrm{e}}$)
    \begin{equation}\label{eq:offset_cost_indication}
        \alpha_{\mathrm{lo}}^T(\vert r \vert_{r_{\mathrm{d}}}) \le T(r) \le \alpha_{\mathrm{up}}^T(\vert r \vert_{r_{\mathrm{d}}}).
    \end{equation}
\end{assumption}
For example, $T(r) = \Vert r - r_{\mathrm{e}} \Vert^2_S - \bar{T}$ with $S\succ 0$ and $\bar{T} = \Vert r_{\mathrm{d}} - r_{\mathrm{e}} \Vert^2_S$ satisfies this on bounded constraint sets, since $T(r_{\mathrm{d}}) = 0$ and $T(r) > 0$ for all $r_{\mathrm{d}} \neq r\in\mathcal{R}$ ($r_{\mathrm{d}}$ is unique).
For simplicity, we assume via~\eqref{eq:offset_cost_indication} $T(r_{\mathrm{d}})=0$  w.l.o.g. (adding constants does not change the MPC optimisation problem).
In particular, $r_{\mathrm{d}}$ (and $\bar{T}$) need not be known.
Moreover, we do not consider output tracking (cf.~\cite{Limon2018,Koehler2020b,Soloperto2022}), but conjecture that our results are transferable.
In this case, Assumption~\ref{asm:offset_cost_indication} would need to hold with $T$ penalising an artificial output.

We consider the case where the stage cost is designed with essentially two purposes.
First, it is employed to set up a stabilising MPC scheme.
Second, it encodes a performance objective, e.g. by penalising large inputs, prioritising use of one input over the other or the deviation in more relevant states.
Hence, we are not only interested in stability of the desired reference, but also in how it is achieved, measured by the stage cost.
More precisely, we aim to compare the performance of (optimal) input sequences $u$, defined by
$
    J_K^{\mathrm{d}}(x, u) = \sum_{k=0}^{K-1} \ell(x_u(k, x), u(k), r_{\mathrm{d}}).
$

\section{MPC for tracking}
\subsection{MPC for tracking scheme}
In this section, we propose changes to the standard MPC for tracking scheme for nonlinear systems~\cite{Limon2018} and show that the best reachable steady state is asymptotically stabilised.

We make the following assumptions on the stage cost, the first two similar to~\cite[Assm. 1]{Soloperto2022}.
The first states that the stage cost provides an adequate measure of the state's distance to the artificial steady state, and the second and third allow for a comparison of stage costs with different references.
\begin{assumption}\label{asm:stage_cost_lower_and_upper_bound}
    There exist $c_1^{\ell}, c_2^{\ell} > 0$ such that for all $(x,u) \in \mathcal{Z}$ and $r = (x_r, u_r) \in \mathcal{R}$
    \begin{equation}\label{eq:stage_cost_lower_and_upper_bound}
        c_1^{\ell} \vert x \vert_{x_r}^2 \le \bar{\ell}(x, r) \le c_2^{\ell} \vert x \vert_{x_r}^2
    \end{equation}
    with
    $
        \bar{\ell}(x, r) = \min_{u \in U\,\text{s.t.}\,(x,u)\in\mathcal{Z}}\ell(x, u, r).
    $
\end{assumption}
\begin{assumption}\label{asm:stage_cost_difference_bound}
    There exist $c_3^{\ell}, c_4^{\ell} > 0$ such that for any $r_{1}, r_{2} \in \mathcal{R}$ and $(x, u) \in \mathcal{Z}$
    \begin{equation}\label{eq:stage_cost_difference_bound}
        \ell(x, u, r_{1}) \le c_3^{\ell} \ell(x, u, r_{2})+ c_4^{\ell} \vert r_{1} \vert_{r_{2}}^2.
    \end{equation}
\end{assumption}
\begin{assumption}\label{asm:stage_cost_difference_bound_with_linear_term}
    There exist $c_5^{\ell}, c_6^{\ell} > 0$ such that for any $r_{1}, r_{2} \in \mathcal{R}$ and $(x, u) \in \mathcal{Z}$
    \begin{equation}\label{eq:stage_cost_difference_bound_with_linear_term}
        \ell(x, u, r_{1}) \le \ell(x, u, r_{2}) + c_5^{\ell} \vert r_{1} \vert_{r_{2}}^2 + c_6^{\ell} \vert r_{1} \vert_{r_{2}}.
    \end{equation}
\end{assumption}
\begin{remark}
    For example, the common quadratic stage cost $\ell(x, u, r) = \Vert x - x_r \Vert_Q^2 + \Vert u - u_r \Vert_R^2$ with $Q,R \succ 0$
    satisfies Assumptions~\ref{asm:stage_cost_lower_and_upper_bound}--\ref{eq:stage_cost_difference_bound_with_linear_term} on bounded constraint sets.
    We refer to~\cite{Soloperto2022} for Assumption~\ref{asm:stage_cost_difference_bound} while Assumption~\ref{asm:stage_cost_difference_bound_with_linear_term} follows from
    $
        \ell(x, u, r_1) 
        = \Vert x - x_{r_1} + x_{r_2} - x_{r_2} \Vert_Q^2 + \Vert u - u_{r_1} + u_{r_2} - u_{r_2} \Vert_R^2 
        = \ell(x, u, r_2) + \Vert x_{r_2} - x_{r_1} \Vert_Q^2 + \Vert u_{r_1} - u_{r_2} \Vert_R^2
        + 2(x - x_{r_1})^\top Q (x_{r_1} - x_{r_2}) + 2( u - u_{r_1})^\top R (u_{r_1} - u_{r_2}),
    $
    the Cauchy-Schwarz inequality, $c_5^{\ell} = \max(\mu_{\mathrm{max}}(Q), \mu_{\mathrm{max}}(R))$ and $c_6^{\ell} = 2\sqrt{2}c_5^{\ell}\max(\sup_{x\in X, r\in\mathcal{R}} \vert x \vert_{x_r}, \sup_{u\in U, r\in\mathcal{R}} \vert u \vert_{u_r})$.
\end{remark}

We use suitable terminal ingredients to ensure recursive feasibility and stability of the closed-loop system.
\begin{assumption}\label{asm:terminal_ingredients}
    There exist $c_{\mathrm{b}}, c_{\mathrm{f}} > 0$, a control law
    $
    k_{\mathrm{f}}: X \times \mathcal{R} \to \mathbb{R}^m
    $, a continuous terminal cost
    $
    V^{\mathrm{f}}: X \times \mathcal{R} \to \mathbb{R}_{\ge 0}
    $ and compact terminal sets
    $
    \mathcal{X}^{\mathrm{f}}(r) \subseteq X
    $ such that for any $r = (x_r, u_r)\in \mathcal{R}$ and any $x \in \mathcal{X}^{\mathrm{f}}(r)$
    \begin{subequations}
        \begin{align}
            x^{+} = f(x, k_{\mathrm{f}}(x, r)) &\in \mathcal{X}^{\mathrm{f}}(r),
            \label{eq:terminal_invariance}
                \\
            V^{\mathrm{f}}(x^{+}, r) - V^{\mathrm{f}}(x, r) 
            &\le  - \ell(x, k_{\mathrm{f}}(x, r), r), \label{eq:terminal_cost_decrease}
                \\
            (x, k_{\mathrm{f}}(x, r)) &\in \mathcal{Z},
            \label{eq:terminal_constraint_satisfaction}
                \\
            \mathcal{B}_{c_{\mathrm{b}}}(x_r) &\subseteq \mathcal{X}^{\mathrm{f}}(r), \label{eq:non-trivial_terminal_set}
                \\
            V^{\mathrm{f}}(x, r) &\le c_{\mathrm{f}} \bar{\ell}(x, r). \label{eq:terminal_cost_upper_bound}
        \end{align}
    \end{subequations}
\end{assumption}
Conditions~\eqref{eq:terminal_invariance}--\eqref{eq:terminal_cost_upper_bound} (combined with~\eqref{eq:stage_cost_lower_and_upper_bound}) are standard in MPC for tracking (cf.~\cite[Assm. 3]{Limon2018}, \cite[Assm. 2]{Koehler2020b}, \cite[Assm. 2]{Koehler2020a}).
The non-degeneracy condition~\eqref{eq:non-trivial_terminal_set} (cf.~\cite[Lem. 5]{Koehler2020b}, \cite[Assm. 2]{Koehler2020a}) allows us to uniformly conclude that $x\in\mathcal{X}^{\mathrm{f}}(r)$ if $\vert x \vert_{x_{r}} \le c_{\mathrm{b}}$. See, e.g.~\cite{Koehler2020a} for designing suitable ingredients.

The optimisation problem used in the MPC for tracking scheme at time $t$ with measured state $x = x(t)$ is
\begin{equation}\label{eq:MPC_for_tracking}
        V_N(x) = \min_{\mathclap{\substack{
            r(x) \in \mathcal{R},
                \\
            u(\cdot \vert t) \in \mathbb{U}^N_{\mathcal{X}^{\mathrm{f}}(r(x))}(x)
        }}} 
        J_N^{\mathrm{tc}}(x, u(\cdot \vert t), r(x)) + \lambda(N) T(r(x))
\end{equation}
where the tracking part of the objective function is
$
    J_N^{\mathrm{tc}}(x, u, r) = \sum_{k=0}^{N-1} \ell(x_u(k, x), u(k), r) + V^\mathrm{f}(x_u(N, x), r)
$,
$N\in\mathbb{N}_0$ is the prediction horizon
and $\lambda:\mathbb{N}_0 \to \mathbb{R}_{\ge 0}$ is a scaling function.
We denote the optimal input sequence by $u_N^*(\cdot\vert t)$ and the optimal artificial reference by $r_N^*(x(t))$.
The set of all states for which~\eqref{eq:MPC_for_tracking} is feasible is denoted by $\mathcal{X}_N$.

Compared to usual formulations of MPC for tracking, see e.g.~\cite{Limon2018}, we include a scaling function $\lambda(N)$.
By rescaling $T$ for fixed $N$ or setting $\lambda(N) \equiv 1$, the usual formulation is recovered for all $N\in\mathbb{N}_0$.
We will, however, show that scaling the offset cost with the horizon is essential, when the latter is varied for performance.
For this purpose, we assume that the scaling function is at least linear and unbounded.
\begin{assumption}\label{asm:superlinear_scaling}
    The scaling function $\lambda$ in~\eqref{eq:MPC_for_tracking} satisfies $\lambda(N) \ge N$ for all $N\in\mathbb{N}_0$ and $\lambda(0) \ge 1$.
\end{assumption}

The MPC for tracking scheme is as follows:
At each time step, we measure $x(t)$, solve~\eqref{eq:MPC_for_tracking} and apply the first part of the optimal input trajectory, $\mu_{N}(x(t)) = u_N^*(0 \vert t)$, leading to the closed-loop system
\begin{equation}\label{eq:closed-loop_system}
    x(t+1) = f(x(t), \mu_{N}(x(t))).
\end{equation}

\subsection{Exponential stability}
The main idea in showing stability of the best reachable reference is that if the offset cost is non-zero for the current artificial reference, there exists a better one in the next time step that reduces the offset cost more than it increases the tracking cost.
Thus, the artificial reference will be at least incrementally moved towards a minimiser of $T$, with the closed loop following due to the tracking part. 
The following assumption, adapted from~\cite[Assm. 3]{Soloperto2022}, provides us with a reference with a lower offset cost.
\begin{assumption}\label{asm:better_candidate_reference}
    There exist $c_{1}^\mathrm{r}, c_{2}^\mathrm{r} > 0$ such that for any $r=(x_r,u_r)\in\mathcal{R}$, and any $\theta \in [0,1]$, there exists $\hat{r}=(x_{\hat{r}}, u_{\hat{r}})\in\mathcal{R}$ with
    \begin{align}
        \vert \hat{r} \vert_{r} &\le c_{1}^\mathrm{r} \theta \vert r \vert_{r_{{\mathrm{d}}}}, \label{eq:candidate_reference_upper_bound}
            \\
        T(\hat{r}) - T(r) &\le -c_{2}^\mathrm{r} \theta \vert r \vert_{r_{{\mathrm{d}}}}^2. \label{eq:candidate_reference_decrease}
    \end{align}
\end{assumption}
This assumption generalises usual assumptions in MPC for tracking, cf.~\cite[Assm. 1 \& 2]{Limon2018}, \cite[Assm. 6]{Koehler2020b}, which impose (strong) convexity on the offset cost, convexity of the set of references and a uniqueness condition. See~\cite[Assm. 4]{Soloperto2022} for an extension to a non-convex set of references.

In the next lemma, we show that the tracking part of the objective function is bounded by the stage cost if the state is sufficiently close to the reference (cf.~\cite[Assm. 2]{Koehler2020b}).
\begin{lemma}\label{lem:feasible_inched_candidate}
    Suppose Assumptions~\ref{asm:stage_cost_lower_and_upper_bound} and~\ref{asm:terminal_ingredients} hold.
    Consider the following optimisation problem where $\hat{r}$ is a fixed parameter:
    $\min_{u \in \mathbb{U}^N_{\mathcal{X}^{\mathrm{f}}(\hat{r})}(x)} J_N^{\mathrm{tc}}(x, u, \hat{r})$.
    Then, there exists $\varepsilon > 0$ such that for any $\hat{r} \in \mathcal{R}$ and any $\hat{x} \in X$ with $\bar{\ell}(\hat{x}, \hat{r}) \le \varepsilon$, 
    this optimisation problem
    is feasible and its solution $\hat{u}$ satisfies the following bound with $c_{\mathrm{f}}$ from Assumption~\ref{asm:terminal_ingredients}:
        \begin{equation}\label{eq:tracking_value_function_upper_bound}
            J_N^{\mathrm{tc}}(\hat{x}, \hat{u}, \hat{r}) \le c_{\mathrm{f}} \bar{\ell}(\hat{x}, \hat{r}).
        \end{equation}
\end{lemma}
\begin{proof}
    Choose $\varepsilon \le c_1^{\ell}c_{\mathrm{b}}^2$ with $c_1^{\ell}$ from Assumption~\ref{asm:stage_cost_lower_and_upper_bound} and $c_{\mathrm{b}}$ from Assumption~\ref{asm:terminal_ingredients}.
    Then, from~\eqref{eq:stage_cost_lower_and_upper_bound},
    $
        c_{1}^{\ell} \vert \hat{x} \vert_{x_{\hat{r}}}^2 \le \bar{\ell}(\hat{x}, \hat{r}) \le \varepsilon \le c_1^{\ell}c_{\mathrm{b}}^2
    $
    which implies $\vert \hat{x} \vert_{x_{\hat{r}}} \le c_{\mathrm{b}}$, hence, from Assumption~\ref{asm:terminal_ingredients}, $\hat{x} \in \mathcal{B}_{c_{\mathrm{b}}}(x_{\hat{r}}) \subseteq \mathcal{X}^{\mathrm{f}}(\hat{r})$.
    The following steps are standard, cf.~\cite[Prop. 2.35]{Rawlings2020}.
    Within the terminal set $\mathcal{X}^{\mathrm{f}}(\hat{r})$, the terminal controller $k_{\mathrm{f}}$ constitutes a feasible input sequence 
    in the above optimisation problem.
    Thus, from~\eqref{eq:terminal_cost_decrease}, the terminal cost $V^{\mathrm{f}}$ is an upper bound on $J_N^{\mathrm{tc}}(\hat{x}, \hat{u}, \hat{r})$.
    The upper bound~\eqref{eq:tracking_value_function_upper_bound} then follows from~\eqref{eq:terminal_cost_upper_bound}.
\end{proof}

Next, we show that the closed-loop stage cost is an upper bound on the distance of the artificial reference to the best reachable reference.
This has been shown first in~\cite[Prop. 2]{Soloperto2022} for a scheme without terminal ingredients.
\begin{lemma}\label{lem:stage_cost_upper_bounds_reference_distance}
    Let Assumptions~\ref{asm:stage_cost_lower_and_upper_bound},~\ref{asm:stage_cost_difference_bound},~\ref{asm:terminal_ingredients},~\ref{asm:superlinear_scaling} and~\ref{asm:better_candidate_reference} hold.
    Then, there exists $c_{\mathrm{d}}^{\ell} > 0$ such that for any $N \in \mathbb{N}_0$ and $x \in \mathcal{X}_N$,
    \begin{equation}\label{eq:stage_cost_upper_bounds_reference_distance}
        \ell(x, \mu_{N}(x), r_N^*(x)) \ge c_{\mathrm{d}}^{\ell} 
        \vert r_N^*(x) \vert_{r_{{\mathrm{d}}}}^2.
    \end{equation}
\end{lemma}
\begin{proof}
    Let $x \in \mathcal{X}_N$ and assume with $ r_N^* = r_N^*(x)$
    \begin{equation}\label{eq:stage_cost_upper_bounds_reference_distance_reverse}
        \ell(x, \mu_{N}(x), r_N^*) < c_{\mathrm{d}}^{\ell} \vert r_N^* \vert_{r_{{\mathrm{d}}}}^2.
    \end{equation}
    Consider the candidate reference $\hat{r} = (x_{\hat{r}}, u_{\hat{r}})$ from Assumption~\ref{asm:better_candidate_reference}. 
    Since $\mathcal{R}$ is compact, there exists a constant $\delta > 0$ with $\delta = \sup_{r\in\mathcal{R}} \vert r \vert_{r_{\mathrm{d}}}$.
    Then, from Assumption~\ref{asm:stage_cost_difference_bound}
    \begin{align}
        &\bar{\ell}(x, \hat{r}) \le \ell(x, \mu_{N}(x), \hat{r}) 
        \stackrel{\eqref{eq:stage_cost_difference_bound}}{\le} c_3^{\ell} \ell(x, \mu_{N}(x), r_N^*) + c_4^{\ell} \vert \hat{r} \vert_{r_N^*}^2 
            \notag
            \\
        &\stackrel{\mathclap{\eqref{eq:stage_cost_upper_bounds_reference_distance_reverse},\eqref{eq:candidate_reference_upper_bound}}}{<} 
        (c_3^{\ell}c_{\mathrm{d}}^{\ell} + c_4^{\ell}(c_1^{\mathrm{r}})^2 \theta^2) \vert r_N^* \vert_{r_{\mathrm{d}}}^2 \le (c_3^{\ell}c_{\mathrm{d}}^{\ell} + c_4^{\ell}(c_1^{\mathrm{r}})^2 \theta^2) \delta^2 \le \varepsilon
        \label{eq:stage_cost_distance_of_reference_bound}
    \end{align}
    where the last inequality follows by choosing $c_{\mathrm{d}}^{\ell}$ and $\theta$ sufficiently small.
    Then, from Lemma~\ref{lem:feasible_inched_candidate}, there exists an input sequence $\hat{u}$ such that $(\hat{u}, \hat{r})$ is feasible for problem \eqref{eq:MPC_for_tracking} and
    \begin{align}\label{eq:combined_candidate_tracking_cost_bound}
        \hspace{-0.5em}
        J_N^{\mathrm{tc}}(x, \hat{u}, \hat{r})
        \stackrel{\mathclap{\eqref{eq:tracking_value_function_upper_bound}}}{\le} c_{\mathrm{f}} \bar{\ell}(x, \hat{r})
        \stackrel{\mathclap{\eqref{eq:stage_cost_distance_of_reference_bound}}}{<} 
        c_{\mathrm{f}}(c_3^{\ell}c_{\mathrm{d}}^{\ell} + c_4^{\ell}(c_1^{\mathrm{r}})^2 \theta^2) \vert r_N^* \vert_{r_{\mathrm{d}}}^2. 
    \end{align}
    Note that $J_N^{\mathrm{tc}}(x, u_N^*, r_N^*) \ge 0$.
    Finally, we arrive at
    \begin{align*}
        &J_N^{\mathrm{tc}}(x, \hat{u}, \hat{r}) - J_N^{\mathrm{tc}}(x, u_N^*, r_N^*) + \lambda(N)(T(\hat{r}) - T(r_N^*)) 
            \\
        &\stackrel{\eqref{eq:combined_candidate_tracking_cost_bound},\eqref{eq:candidate_reference_decrease}}{<} 
        (c_{\mathrm{f}}(c_3^{\ell}c_{\mathrm{d}}^{\ell} + c_4^{\ell}(c_1^{\mathrm{r}})^2 \theta^2) - c_2^{\mathrm{r}}\theta) \vert r_N^* \vert_{r_{\mathrm{d}}}^2 = - \bar{c}_{\mathrm{d}}^{\ell} \vert r_N^* \vert_{r_{\mathrm{d}}}^2
    \end{align*}
    with $\bar{c}_{\mathrm{d}}^{\ell} = - (c_{\mathrm{f}}c_3^{\ell}c_{\mathrm{d}}^{\ell} + (c_{\mathrm{f}}c_4^{\ell}(c_1^{\mathrm{r}})^2 \theta - c_2^{\mathrm{r}})\theta)$, where we used that $\lambda(N) \ge 1$.
    By first fixing $\theta < c_2^{\mathrm{r}}(c_{\mathrm{f}}c_4^{\ell}(c_1^{\mathrm{r}})^2)^{-1}$
    and then choosing $c_{\mathrm{d}}^{\ell} < (c_{\mathrm{f}}c_4^{\ell}(c_1^{\mathrm{r}})^2 \theta - c_2^{\mathrm{r}})\theta(c_{\mathrm{f}}c_3^{\ell})^{-1}$, we have $\bar{c}_{\mathrm{d}}^{\ell} > 0$.
    Hence, the cost of the candidate $(\hat{u}, \hat{r})$ is smaller than that of the optimal $(u_N^*, r_N^*)$, which is a contradiction.
\end{proof}

In the following, we show stability of the closed loop~\eqref{eq:closed-loop_system}, using the value function $V_N$
as a Lyapunov candidate.
\begin{theorem}\label{thm:exponential_stability}
    Let Assumptions~\ref{asm:stage_cost_lower_and_upper_bound},~\ref{asm:stage_cost_difference_bound},~\ref{asm:terminal_ingredients},~\ref{asm:superlinear_scaling} and~\ref{asm:better_candidate_reference} hold, and let $N\in\mathbb{N}_0$.
    Assume that~\eqref{eq:MPC_for_tracking} is feasible at time $t=0$, i.e. $x(0) \in \mathcal{X}_N$.
    Then,~\eqref{eq:MPC_for_tracking} is feasible for all $t\in\mathbb{N}_0$ and $x_{\mathrm{d}}$ is exponentially stable for the closed loop~\eqref{eq:closed-loop_system} with region of attraction $\mathcal{X}_N$. 
    That is, there exist $c_{\mathrm{exp}} > 0$ and $\gamma_{\mathrm{exp}} \in (0,1)$ such that for all $x \in \mathcal{X}_N$ and $t \in \mathbb{N}_0$
    \begin{equation}\label{eq:exponential_stability}
        \vert x_{\mu_N}(t,x) \vert_{x_{\mathrm{d}}} \le c_{\mathrm{exp}} \vert x \vert_{x_{\mathrm{d}}} \gamma_{\mathrm{exp}}^t.
    \end{equation}
\end{theorem}
\begin{proof}
    First, we show recursive feasibility.
    Let~\eqref{eq:MPC_for_tracking} be feasible at $t \in \mathbb{N}_0$.
    Then, at $t+1$, standard arguments using Assumption~\ref{asm:terminal_ingredients} (cf., e.g.~\cite[Lem.~5.10]{Gruene2017}) yield a feasible candidate $(\tilde{u}(\cdot \vert t+1), \tilde{r}(t+1))$ with $\tilde{r}(t+1) = r_N^*(x(t))$ and $\tilde{u}(\cdot \vert t+1) = (u_N^*(1 \vert t), \dots, u_N^*(N-1 \vert t), k_{\mathrm{f}}(x_{u_N^*(\cdot \vert t)}(N), r_N^*(x(t))))$. Feasibility of~\eqref{eq:MPC_for_tracking} for any $t \in \mathbb{N}_0$ follows by induction.

    Second, we show an upper bound of the value function.
    Let $x \in X$ such that $c_2^{\ell} \vert x \vert_{x_\mathrm{d}}^2 \le \varepsilon$.
    Then, from~\eqref{eq:stage_cost_lower_and_upper_bound}, $\bar{\ell}(x, r_{\mathrm{d}}) \le \varepsilon$ and hence, from Lemma~\ref{lem:feasible_inched_candidate}, $x \in \mathcal{X}_N$ and
    $
        V_N(x) 
        \le J_N^{\mathrm{tc}}(x, u_N^*(\cdot \vert t), r_{\mathrm{d}}) 
        \le c_{\mathrm{f}} \bar{\ell}(x, r_{\mathrm{d}})
        \le c_{\mathrm{f}}c_2^{\ell} \vert x \vert_{x_{\mathrm{d}}}^2,
    $
    where the second inequality follows from \eqref{eq:tracking_value_function_upper_bound} and the third from \eqref{eq:stage_cost_lower_and_upper_bound}.
    This local upper bound and compact constraints imply that for all $N\in\mathbb{N}_0$ there exists $a_{\mathrm{up}} > 0$ such that 
    \begin{equation}\label{eq:value_function_upper_bound}
        V_N(x) \le a_{\mathrm{up}}\vert x \vert_{x_{\mathrm{d}}}^2
    \end{equation}
    for all $x \in \mathcal{X}_N$, cf.~\cite[Prop. 2.16]{Rawlings2020}.

    Third, we show a lower bound of the value function. Define $a_{\mathrm{lo}} =  0.25 \min\{c_1^\ell, c_{\mathrm{d}}^\ell \}$, then for all $x\in\mathcal{X}_N$
    \begin{align}\label{eq:value_function_lower_bound}
        &V_N(x) \ge \ell(x, u_N^*(0 \vert t), r_N^*(x))
            \notag
            \\
        &\stackrel{\eqref{eq:stage_cost_lower_and_upper_bound},\eqref{eq:stage_cost_upper_bounds_reference_distance}}{\ge} 
        \frac{c_1^\ell}{2}\vert x \vert_{x_{r_N^*}(x)}^2 
        + \frac{c_{\mathrm{d}}^\ell}{2} \vert x_{r_N^*(x)} \vert_{x_{\mathrm{d}}}^2 \ge a_{\mathrm{lo}}\vert x \vert_{x_{\mathrm{d}}}^2.
    \end{align}

    Fourth, we show a decrease in the value function.
    Consider again $(\tilde{u}(\cdot \vert t+1), \tilde{r}(t+1))$ from above.
    Then,
    \begin{align}
        &V_N(x(t+1)) - V_N(x(t))  
            \notag
            \\
        &\le J_N^{\mathrm{tc}}(x(t+1), \tilde{u}(\cdot \vert t+1), \tilde{r}(t+1)) + \lambda(N)T(\tilde{r}(t+1))
            \notag
            \\
        &\phantom{\le{}} - J_N^{\mathrm{tc}}(x(t), u_N^*(\cdot \vert t), r_N^*(x(t))) - \lambda(N)T(r_N^*(x(t)))
            \notag
            \\
        &= \sum_{k=0}^{N-2} \ell(x_{u_N^*(\cdot \vert t)}(k+1, x(t)), u_N^*(k+1 \vert t), r_N^*(x(t)))
            \notag
            \\
        &\phantom{\le{}} + \ell(x_{u_N^*(\cdot \vert t)}(N, x(t)), k_{\mathrm{f}}(x_{u_N^*(\cdot \vert t)}(N, x(t)), r_N^*(x(t))))
            \notag
            \\
        &\phantom{\le{}} + V^{\mathrm{f}}\left(x_{\tilde{u}(\cdot \vert t+1)}(N, x(t+1)), r_N^*(x(t))\right)
            \notag
            \\
        &\phantom{\le{}} - \sum_{k=0}^{N-1} \ell(x_{u_N^*(\cdot \vert t)}(k, x(t)), u_N^*(k \vert t), r_N^*(x(t)))
            \notag
            \\
        &\phantom{\le{}} - V^{\mathrm{f}}(x_{u_N^*(\cdot \vert t)}(N, x(t)), r_N^*(x(t)))
            \notag    
            \\
        &\stackrel{\eqref{eq:terminal_cost_decrease}}{\le} - \ell({x}(t), u_N^*(0 \vert t), r_N^*(x(t))) \stackrel{\eqref{eq:value_function_lower_bound}}{\le} 
        - a_{\mathrm{lo}}\vert x(t) \vert_{x_{\mathrm{d}}}^2
        \label{eq:value_function_decrease}
    \end{align}
    where the first equality follows from the definition of $(\tilde{u}(\cdot \vert t+1), \tilde{r}(t+1))$.
    Exponential stability then follows from~\eqref{eq:value_function_upper_bound},~\eqref{eq:value_function_lower_bound} and~\eqref{eq:value_function_decrease} using standard arguments (cf.~\cite[Thm.~B.19]{Rawlings2020}).
\end{proof}
\begin{remark}
    As stated in~\cite[Assm. 4]{Limon2018}, \cite[Prop. 4]{Koehler2020b}, if the system is locally uniformly finite time controllable, Assumption~\ref{asm:terminal_ingredients} can be replaced in Theorem~\ref{thm:exponential_stability} by a terminal equality constraint, i.e. $\mathcal{X}^{\mathrm{f}}(r)=\{x_{r}\}$ and $V^{\mathrm{f}}(x, r) \equiv 0$. 
\end{remark}

\section{Transient performance estimate}
In this section, we will derive an estimate of the closed-loop performance.
Appreciate that from~\eqref{eq:value_function_decrease} we obtain
\begin{align}
    &\ell({x}_{\mu_{N}}(k), \mu_{N}({x}_{\mu_{N}}(k)), r_N^*({x}_{\mu_{N}}(k))) 
        \notag    
        \\
    &\le V_N({x}_{\mu_{N}}(k)) - V_N({x}_{\mu_{N}}(k+1)),
    \label{eq:value_function_decrease_recast}
\end{align}
for any $k \in \mathbb{N}_0$.
From this, we will derive a performance bound on the closed loop by using Assumption~\ref{asm:stage_cost_difference_bound_with_linear_term}, and by finding a suitable upper bound on the value function $V_N$.

We will do so by comparing the solution of the MPC for tracking optimisation problem~\eqref{eq:MPC_for_tracking} with the standard MPC optimisation problem's solution with respect to the best reachable reference.
The latter is given by
\begin{equation}\label{eq:standard_MPC}
    W_N(x(t)) = \min_{\mathclap{u_{\mathrm{std}}(\cdot \vert t) \in \mathbb{U}^N_{\mathcal{X}^{\mathrm{f}}(r_{\mathrm{d}})}(x(t))}} 
    J_N^{\mathrm{tc}}(x(t), u_{\mathrm{std}}(\cdot \vert t), r_{\mathrm{d}}).
\end{equation}
We denote the set of states where~\eqref{eq:standard_MPC} is feasible by $\mathcal{X}^{\mathrm{std}}_N$.
Similar to before, the standard MPC scheme is as follows:
At each time step, we measure $x(t)$, solve~\eqref{eq:standard_MPC} and apply the first part of the optimal input, $\nu_{N}(x(t)) = u_{\mathrm{std}}^*(0 \vert t)$.

Note that setting up~\eqref{eq:standard_MPC} requires knowledge of the best reachable reference, which is not the case in~\eqref{eq:MPC_for_tracking}.
Compared to standard MPC (20), MPC for tracking (6) contains $n+m$ additional decision variables and the class of the two optimisation problems is the same, e.g. if (20) is a convex quadratic program, then (6) can be set up to be one as well~\cite{Limon2008}.

\begin{proposition}\label{prop:standard_MPC_stability}
    Let Assumptions~\ref{asm:stage_cost_lower_and_upper_bound} and~\ref{asm:terminal_ingredients} hold.
    Then, $x_{\mathrm{d}}$ is exponentially stable for the closed loop using the standard MPC scheme for all $x\in\mathcal{X}^{\mathrm{std}}_N$,
    i.e. there
    exist $\gamma_{\mathrm{std}} \in (0,1)$ and
    $c_{\mathrm{std}} > 0$ such that for all $t \in \mathbb{N}_0$ we have
    $\vert x_{\nu_N}(t, x) \vert_{x_{\mathrm{d}}} \le c_{\mathrm{std}} \vert x \vert_{x_{{\mathrm{d}}}} \gamma_{\mathrm{std}}^t$. 
    Furthermore, for all $N\in\mathbb{N}_0$ there exists $\alpha_W \in \mathcal{K}_{\infty}$ such that $W_N(x) \le \alpha_W(\vert x \vert_{x_{\mathrm{d}}})$ for all $x\in\mathcal{X}^{\mathrm{std}}_N$.
\end{proposition}
We omit the proof, which is standard (see e.g.~\cite{Rawlings2020,Gruene2017}).

The following lemma states that the value function of a standard MPC optimisation problem~\eqref{eq:standard_MPC} is an upper bound of the value function of MPC for tracking~\eqref{eq:MPC_for_tracking}, if the horizon is chosen sufficiently long. 
As part of the proof, we show that the terminal set of the best reachable reference can be reached in uniform finite time from any other terminal set.
\begin{lemma}\label{lem:standard_MPC_upper_bounds_value_function}
    Let Assumptions~\ref{asm:offset_cost_indication}--\ref{asm:terminal_ingredients} and \ref{asm:better_candidate_reference} hold.
    Then, for all $\widetilde{N} \in \mathbb{N}_0 $ with $\mathcal{X}_{\widetilde{N}} \neq \emptyset$, there exists $\bar{N} \in \mathbb{N}_0$ such that for all $x \in \mathcal{X}_{\widetilde{N}}$ and $N \ge \bar{N}$, the problem~\eqref{eq:standard_MPC} has a solution and
    \begin{equation}\label{eq:standard_MPC_upper_bounds_MPC_for_tracking}
        V_N(x) \le W_{N}(x) \le W_{\bar{N}}(x).
    \end{equation}
\end{lemma}
\begin{proof}
    We start by showing that $\mathcal{X}^{\mathrm{f}}(r_{\mathrm{d}})$ is uniformly reachable from any $x \in \cup_{r\in\mathcal{R}}\mathcal{X}^{\mathrm{f}}(r)$.
    Let $r\in\mathcal{R}$ and $x\in\mathcal{X}^{\mathrm{f}}(r)$.
    Consider $u_{\mathrm{f}}^r(k) = k_{\mathrm{f}}(x_{u_{\mathrm{f}}^r}(k, x), r)$.
    From Assumptions~\ref{eq:stage_cost_lower_and_upper_bound} and~\ref{asm:terminal_ingredients}, we get for all $k\in\mathbb{N}_0$,
    $c_1^{\ell} \vert x \vert_{x_r}^2 
    \stackrel{\eqref{eq:stage_cost_lower_and_upper_bound},\eqref{eq:terminal_cost_decrease}}{\le} 
    V^{\mathrm{f}}(x,r) 
    \stackrel{\eqref{eq:terminal_cost_upper_bound}, \eqref{eq:stage_cost_lower_and_upper_bound}}{\le}
    c_{\mathrm{f}} c_2^{\ell} \vert x \vert_{x_r}^2$ and 
    $V^{\mathrm{f}}(x_{u_{\mathrm{f}}^r}(k+1), r) - V^{\mathrm{f}}(x_{u_{\mathrm{f}}^r}(k), r) 
    \stackrel{\eqref{eq:terminal_cost_decrease},\eqref{eq:stage_cost_lower_and_upper_bound}}{\le} -c_1^{\ell} \vert x_{u_{\mathrm{f}}^r}(k) \vert_{x_r}^2$.
    Therefore,
    $
        \vert x_{u_{\mathrm{f}}^r}(k, x) \vert_{x_r}^2 \le \tilde{c}_{\mathrm{f}} \vert x \vert_{x_r}^2 \gamma_{\mathrm{f}}^k \le \tilde{c}_{\mathrm{f}} \delta_x \gamma_{\mathrm{f}}^k
    $
    with $\tilde{c}_{\mathrm{f}} = \frac{c_{\mathrm{f}}c_2^{\ell}}{c_1^{\ell}}$, $\gamma_\mathrm{f} = 1 - \tilde{c}_{\mathrm{f}}^{-1} < 1$ and $\delta_x = \sup_{x\in X} \vert x \vert_{x_r}^2$.
    Hence, there exists a finite $\tau$, independent of $r$, such that $\vert x_{u_{\mathrm{f}}^r}(\tau, x) \vert_{x_r} \le \frac{c_{\mathrm{b}}}{2}$.
    In one case, $\vert x_r \vert_{x_{\mathrm{d}}} \le \vert r \vert_{r_{\mathrm{d}}} \le \frac{c_{\mathrm{b}}}{2}$. 
    Then, $\vert x_{u_{\mathrm{f}}^r}(\tau, x) \vert_{x_{\mathrm{d}}} \le \vert x_{u_{\mathrm{f}}^r}(\tau, x) \vert_{x_{\mathrm{r}}} + \vert x_{\mathrm{r}} \vert_{x_{\mathrm{d}}} \le c_{\mathrm{b}}$ which implies $u_{\mathrm{f}}^r \in \mathbb{U}^{N}_{\mathcal{X}^{\mathrm{f}}(r_{\mathrm{d}})}(x)$ for any $N \ge \tau$.
    In the other case, $\delta_{\mathrm{r}} \ge \vert r \vert_{r_{\mathrm{d}}} > \frac{c_{\mathrm{b}}}{2}$, where $\delta_{\mathrm{r}} = \sup_{r,\hat{r}\in\mathcal{R}} \vert r \vert_{\hat{r}}$. 
    Pick $\theta$ and $\hat{r}$ from Assumption~\ref{asm:better_candidate_reference}.
    Then, from~\eqref{eq:candidate_reference_upper_bound}, $\vert x_{u_{\mathrm{f}}^r}(\tau, x) \vert_{x_{\hat{r}}} \le \vert x_{u_{\mathrm{f}}^r}(\tau, x) \vert_{x_r} + \vert r \vert_{\hat{r}} \le c_{\mathrm{b}}$ if 
    $\theta\le\frac{c_{\mathrm{b}}}{2\delta_{\mathrm{r}} c_1^{\mathrm{r}}}$. 
    Hence, for all $k \ge \tau$, $x_{u_{\mathrm{f}}^r}(k, x) \in \mathcal{X}^{\mathrm{f}}(\hat{r})$.
    From~\eqref{eq:candidate_reference_decrease}, $T(\hat{r}) \le T(r) - c_2^{\mathrm{r}}\theta\vert r\vert_{r_{\mathrm{d}}}^2 < T(r) - c_2^{\mathrm{r}}\theta\frac{c_{\mathrm{b}}^2}{4}$.
    With Assumption~\ref{asm:offset_cost_indication}, by repeatedly applying $u_{\mathrm{f}}^{\hat{r}}$ for each new reference $\hat{r}$,
    we converge into $\mathcal{X}^{\mathrm{f}}(r^\prime)$ with $\vert r^\prime \vert_{r_\mathrm{d}} \le \frac{c_{\mathrm{b}}}{2}$ after $\tau(1 + c_{\tau})$ steps with $c_{\tau} \ge 4\frac{\alpha_{\mathrm{up}}^T(\delta_{\mathrm{r}}) - {\alpha_{\mathrm{lo}}^T}^{-1}(\frac{c_{\mathrm{b}}}{2})}{c_2^{\mathrm{r}}\theta c_{\mathrm{b}}^2}$. 
    Then, $\tau$ steps of $u_{\mathrm{f}}^{r^\prime}$ bring us into $\mathcal{X}^{\mathrm{f}}(r_{\mathrm{d}})$, where the input can be feasibly extended to an arbitrary length by Assumption~\ref{asm:terminal_ingredients}.
    In summary, for any $r\in\mathcal{R}$ and $x\in\mathcal{X}^{\mathrm{f}}(r)$, $\mathbb{U}^N_{\mathcal{X}^{\mathrm{f}}(r_{\mathrm{d}})}(x) \neq \emptyset$ for all $N \ge \tau(2 + c_{\tau})$.
    
    Finally, since $x \in \mathcal{X}_{\widetilde{N}}$, there exists $\tilde{r} \in \mathcal{R}$ and $\tilde{u} \in \mathbb{U}_{\mathcal{X}^{\mathrm{f}}(\tilde{r})}^{\widetilde{N}}(x)$ such that $x_{\tilde{u}}(\widetilde{N},x) \in \mathcal{X}^{\mathrm{f}}(\tilde{r})$.
    Hence, as shown above, there exists $u_{\tau} \in \mathbb{U}_{\mathcal{X}^{\mathrm{f}}(r_{\mathrm{d}})}^{N_{\tau}}(x_{\tilde{u}}(\widetilde{N},x))$ with $N_{\tau} = \tau(2 + c_{\tau})$.
    Thus, $u_N = (\tilde{u}, u_{\tau}) \in \mathbb{U}_{\mathcal{X}^{\mathrm{f}}(r_{\mathrm{d}})}^{N}(x)$, for any $N \ge \bar{N} = \widetilde{N} + N_{\tau}$, is a feasible candidate in~\eqref{eq:standard_MPC}.
    This in turn implies that there exists $(u_{\mathrm{std}}^*, r_{\mathrm{d}})$ for all $N \ge \bar{N}$ and $x \in \mathcal{X}_{\widetilde{N}}$, and it is a feasible solution in~\eqref{eq:MPC_for_tracking} with $T(r_{{\mathrm{d}}}) = 0$.
    Hence,
    $
    V_N(x) \le J_N^{\mathrm{tc}}(x, u_{\mathrm{std}}^*, r_{\mathrm{d}}) + \lambda(N)T(r_{{\mathrm{d}}}) = W_{N}(x)
    $
    which is the first inequality of the claim.
    The second is a standard result with $N \ge \bar{N}$, see e.g.~\cite[Lem.~5.12]{Gruene2017}.
\end{proof}

This makes it possible to use well established results on performance of MPC~{\cite[Thm. 8.22]{Gruene2017}}.
\begin{proposition}\label{prop:standard_MPC_value_function_performance_bound}
    Let Assumptions~\ref{asm:stage_cost_lower_and_upper_bound} and~\ref{asm:terminal_ingredients} hold.
    Then, for all $N\in\mathbb{N}_0$ there exist $\delta_1, \delta_2 \in \mathcal{L}$ such that for all $x \in \mathcal{X}_N$
    \begin{equation}\label{eq:standard_MPC_value_function_performance_bound}
        W_N(x) \le \inf_{u\in\mathbb{U}^K_{\mathcal{B}_{\kappa}(x_{{\mathrm{d}}})}(x)} J_K^{\mathrm{d}}(x,u) + \delta_1(N) + \delta_2(K) 
    \end{equation}
    with $\kappa = c_{\mathrm{std}}\vert x \vert_{x_{\mathrm{d}}} \gamma_{\mathrm{std}}^K$ and $c_{\mathrm{std}}, \gamma_{\mathrm{std}}$ from Proposition~\ref{prop:standard_MPC_stability}.
\end{proposition}
\begin{proof}
    All assumptions in~{\cite[Thm. 8.22]{Gruene2017}} are satisfied, and the claim follows from its proof.
\end{proof}

We are now able to state our first main result:
\begin{theorem}\label{thm:transient_performance_bound}
    Let Assumptions~\ref{asm:offset_cost_indication}--\ref{asm:better_candidate_reference} hold.
    Then, for any $\widetilde{N} \in \mathbb{N}_0$, there exist $\delta_1, \delta_2 \in \mathcal{L}$, $\bar{N}\in \mathbb{N}$, such that for all $x \in \mathcal{X}_{\widetilde{N}}$, $N \ge \bar{N}$ and $K\in\mathbb{N}_0$, with $x_{\mu_{N}}(k) = x_{\mu_{N}}(k,x)$,
    \begin{align}\label{eq:transient_performance_bound}
        &J_K^{\mathrm{d}}(x, \mu_N) = \sum_{k=0}^{K-1} \ell(x_{\mu_{N}}(k), \mu_{N}({x}_{\mu_{N}}(k)), r_{{\mathrm{d}}}) 
            \notag
            \\
        &\le \inf_{u\in\mathbb{U}^K_{\mathcal{B}_{\kappa}(x_{{\mathrm{d}}})}(x)} J_K^{\mathrm{d}}(x,u) + \delta_1(N) + \delta_2(K) - V_N({x}_{\mu_{N}}(K)) 
            \notag
            \\
        &\phantom{\le{}}
        + \sum_{k=0}^{K-1} c_5^{\ell} \vert r_N^*(x_{\mu_{N}}(k)) \vert_{r_{{\mathrm{d}}}}^2 + c_6^{\ell} \vert r_N^*(x_{\mu_{N}}(k)) \vert_{r_{{\mathrm{d}}}}
    \end{align}
    with $\kappa = c_{\mathrm{std}}\vert x \vert_{x_{\mathrm{d}}} \gamma_{\mathrm{std}}^K$ and $c_{\mathrm{std}}, \gamma_{\mathrm{std}}$ from Proposition~\ref{prop:standard_MPC_stability}.
\end{theorem}
\begin{proof}
    For any $K \in \mathbb{N}_0$, from Assumption~\ref{asm:stage_cost_difference_bound_with_linear_term} and summing up both sides of~\eqref{eq:value_function_decrease_recast}, we get for any $N \ge \widetilde{N}$
    $
        J_K^{\mathrm{d}}(x, \mu_N) = \sum_{k=0}^{K-1} \ell(x_{\mu_{N}}(k), \mu_{N}({x}_{\mu_{N}}(k)), r_{{\mathrm{d}}})\stackrel{\eqref{eq:stage_cost_difference_bound_with_linear_term}}{\le}
        \sum_{k=0}^{K-1} \ell({x}_{\mu_{N}}(k), \mu_{N}({x}_{\mu_{N}}(k)), r_N^*(x_{\mu_{N}}(k)))
        + \sum_{k=0}^{K-1} c_5^{\ell} \vert r_N^*(x_{\mu_{N}}(k)) \vert_{r_{{\mathrm{d}}}}^2 + c_6^{\ell} \vert r_N^*(x_{\mu_{N}}(k)) \vert_{r_{{\mathrm{d}}}}
        \stackrel{\eqref{eq:value_function_decrease_recast}}{\le} 
        V_N(x) - V_N({x}_{\mu_{N}}(K))
        + \sum_{k=0}^{K-1} c_5^{\ell} \vert r_N^*(x_{\mu_{N}}(k)) \vert_{r_{{\mathrm{d}}}}^2 + c_6^{\ell} \vert r_N^*(x_{\mu_{N}}(k)) \vert_{r_{{\mathrm{d}}}}.
    $
    With Lemma~\ref{lem:standard_MPC_upper_bounds_value_function} and Proposition~\ref{prop:standard_MPC_value_function_performance_bound}, the claim follows.
\end{proof}

Compared to the transient performance bound~\eqref{eq:standard_MPC_value_function_performance_bound} for standard MPC (see~\cite[Chapter 8.4]{Gruene2017} for a discussion of~\eqref{eq:standard_MPC_value_function_performance_bound}, $\delta_1$ and $\delta_2$),~\eqref{eq:transient_performance_bound} contains an additional error term, which depends on the structure of the stage cost ($c_5^{\ell}$, $c_6^{\ell}$) and on the shape of the set of equilibria ($\vert r_N^* \vert_{r_{{\mathrm{d}}}}$).
Appreciate that~\eqref{eq:transient_performance_bound} is an \emph{upper bound} and thus does not imply that, for a fixed horizon, a larger offset cost necessarily improves the closed-loop performance, although the bound could become smaller.

\section{Asymptotic performance estimate}
Next, we examine whether the error terms in~\eqref{eq:transient_performance_bound} stay bounded if $K\to\infty$ (the number of compared closed-loop time steps), and vanish if $N\to\infty$ (the prediction horizon).
The former follows from exponential stability.
\begin{lemma}\label{lem:convergence_of_series}
    Suppose Assumptions~\ref{asm:stage_cost_lower_and_upper_bound},~\ref{asm:stage_cost_difference_bound},~\ref{asm:terminal_ingredients},~\ref{asm:superlinear_scaling} and~\ref{asm:better_candidate_reference} hold.
    Then, for all $N\in\mathbb{N}_0$, there exists $\rho_{N} \ge 0$ such that for all $x \in \mathcal{X}_N$, with $r_N^*(k) = r_N^*(x_{\mu_{N}}(k,x))$,
    \begin{equation}\label{eq:convergence_artificial_reference_series}
        \sum_{k=0}^{\infty} c_5^{\ell} \vert r_N^*(k) \vert_{r_{\mathrm{d}}}^2 + c_6^{\ell} \vert r_N^*(k) \vert_{r_{\mathrm{d}}} \le \rho_{N} < \infty.
    \end{equation}
\end{lemma}
\begin{proof}
    Let $N\in\mathbb{N}_0$ and $x \in \mathcal{X}_N$.
    Then, for all $k\in\mathbb{N}_0$, abbreviating $r_N^*(k) = r_N^*(x_{\mu_{N}}(k,x))$ and $x(k) = x_{\mu_{N}}(k,x)$
    we have
    $
    c_{\mathrm{d}}^{\ell} \vert r_N^*(k) \vert_{r_{\mathrm{d}}}^2 
        \stackrel{\eqref{eq:stage_cost_upper_bounds_reference_distance}}{\le}
        \ell(x(k), \mu_N(x(k)), r_N^*(k)) 
        \le 
        V_N(x(k))
        \stackrel{\eqref{eq:value_function_upper_bound}}{\le} 
        a_{\mathrm{up}} \vert x(k) \vert_{x_{\mathrm{d}}}^2
        \stackrel{\eqref{eq:exponential_stability}}{\le} 
        a_{\mathrm{up}} c_{\mathrm{exp}}^2 \vert x \vert_{x_{\mathrm{d}}}^2 \gamma_{\mathrm{exp}}^{2k}.
    $
    Therefore, $c_5^{\ell} \vert r_N^*(k) \vert_{r_{\mathrm{d}} }^2 \le  c_5^{\ell} \tilde{c}_{\mathrm{exp}} \gamma_{\mathrm{exp}}^{2k}$ with $\tilde{c}_{\mathrm{exp}} = (c_{\mathrm{d}}^{\ell})^{-1}a_{\mathrm{up}} c_{\mathrm{exp}}^2 \vert x \vert_{x_{\mathrm{d}}}^2$. 
    Since $\gamma_{\mathrm{exp}} \in (0,1)$, $\sum_{k=0}^{\infty} c_5^{\ell} \tilde{c}_{\mathrm{exp}} \gamma_{\mathrm{exp}}^{2k} = \frac{c_5^{\ell} \tilde{c}_{\mathrm{exp}} }{1 - \gamma_{\mathrm{exp}}^{2}}$.
    Additionally, $c_6^{\ell} \vert r_N^*(k) \vert_{r_{\mathrm{d}} } \le  c_6^{\ell} \sqrt{\tilde{c}_{\mathrm{exp}}} \gamma_{\mathrm{exp}}^{k}$ and
    $\sum_{k=0}^{\infty} c_6^{\ell} \sqrt{\tilde{c}_{\mathrm{exp}}} \gamma_{\mathrm{exp}}^{k} = \frac{c_6^{\ell} \sqrt{\tilde{c}_{\mathrm{exp}}}}{1 - \gamma_{\mathrm{exp}}}$.
    The claim follows from the comparison test, $\mathcal{X}_N \subseteq X$ for all $N\in\mathbb{N}_0$ and boundedness of $X$.
\end{proof}

The following shows that scaling the offset cost with the horizon, i.e. keeping it in relation to the $N$ terms of the stage cost, is pivotal for the error terms in~\eqref{eq:transient_performance_bound} to vanish if the horizon grows to infinity. 
These depend on the optimal open-loop artificial reference in each time step.
\begin{example}
    This example shows that the artificial reference does not, in general, converge in \emph{open-loop} to the best reachable one for arbitrarily large horizons, if $\lambda(N)$ is chosen constant, e.g. $\lambda(N) \equiv 1$.
    Consider the task of steering the scalar system
    $
        x(t+1) = x(t) + u(t), \quad x(0) = x_0,
    $
    to the steady state $x_{\mathrm{d}} = 0$ given any initial condition $x_0 \in \mathbb{R}$.
    State and input constraints are set to $(x(t), u(t)) \in [-2x_0, 2x_0] \times [-1, 1]$.
    The stage cost is chosen as $\ell(x,u,r) = (x-x_r)^2 + (u-u_r)^2$ and the offset cost as $T(x_r) = (x_r - x_{\mathrm{d}})^2 = x_r^2$.
    Note that $u_r = 0$.
    Then, the MPC for tracking optimisation problem with (for simplicity) a terminal equality constraint and $\lambda(N) \equiv 1$ is
    $
        \min_{
            u, r
            } 
            \sum_{k=0}^{N-1} (x_u(k) - x_r)^2 + u(k)^2 + x_r^2
    $ subject to $u \in \mathbb{U}^N_{\{x_r\}}(x_0)$ and $r \in \mathcal{R}$. 
    The cost of staying at the initial state, i.e. with $u(k) = 0$, $x_u(k) = x_0$ for all $k$, and $x_r = x_0$, is $x_0^2$.
    Next, consider any steady state between $\frac{x_0}{2}$ and $x_{\mathrm{d}} = 0$, i.e. $x_{\mathrm{h}} = \vartheta\frac{x_0}{2}$ with $\vartheta \in [0, 1]$.
    Clearly, there exists $u_{\mathrm{h}} \in \mathbb{U}_{\{x_{\mathrm{h}}\}}^N$ for any $N$ sufficiently large.
    The cost of $(u_{\mathrm{h}}, x_{\mathrm{h}})$ then satisfies
    $
        \sum_{k=0}^{N-1} (x_{u_{\mathrm{h}}}(k) - x_{\mathrm{h}})^2 + u_{\mathrm{h}}(k)^2 + x_{\mathrm{h}}^2 > (x_0 - x_{\mathrm{h}})^2 + x_{\mathrm{h}}^2
        \ge x_0^2 - 2x_0x_{\mathrm{h}} + 2x_{\mathrm{h}}^2 = x_0^2 - \vartheta x_0^2 + \frac{\vartheta^2}{4}x_0^2 
        \ge x_0^2
    $
    for all $\vartheta \in [0, 1]$.
    Therefore, choosing $x_r = x_0$ is always better than $x_r = x_{\mathrm{h}}$ for arbitrary horizons.
\end{example}

In comparison, a scaled offset cost leads to uniform convergence on each fixed set of feasible states.
\begin{lemma}\label{lem:convergence_artificial_reference}
    Let Assumptions~\ref{asm:offset_cost_indication}--\ref{asm:better_candidate_reference} hold and $M \in \mathbb{N}_0$.
    Then, $\lim_{N\to\infty} \vert r_N^*(x) \vert_{r_{\mathrm{d}}} = 0$ uniformly on $\mathcal{X}_{M}$.
\end{lemma}
\begin{proof}
    Suppose there exists $\kappa > 0$ such that for all $N \in \mathbb{N}_0$ there exists $x \in \mathcal{X}_{M}$ with $\vert r_N^*(x) \vert_{r_{\mathrm{d}}} \ge \kappa$.
    Assumption~\ref{asm:offset_cost_indication} implies $T(r_{N}^*(x)) \ge \alpha_{\mathrm{lo}}^T(\kappa)$.
    From Lemma~\ref{lem:standard_MPC_upper_bounds_value_function}, there exists $\bar{N} \in \mathbb{N}$ such that $V_N(x) \le W_{\bar{N}}(x)$ for all $x \in \mathcal{X}_{M}$ and $N \ge \bar{N}$.
    Hence, since $J_N^{\mathrm{tc}}$ is non-negative,
    $
        \lambda(N) T(r_N^*(x)) \le V_N(x) \le W_{\bar{N}}(x) 
    $.
    Also, $\mathcal{X}_{M} \subseteq X$, and since $X$ is bounded, there exists $\delta_{M} = \sup_{x \in \mathcal{X}_{M}} \vert x \vert_{x_{\mathrm{d}}}$.
    With $\alpha_W$ from Proposition~\ref{prop:standard_MPC_stability} and $T(r_N^*(x)) \ge \alpha_{\mathrm{lo}}^T(\kappa)$ for all $N\in\mathbb{N}_0$, for any $N \ge \bar{N}$, we have 
    $
        \lambda(N) \le \frac{\alpha_W(\delta_{M})}{T(r_N^*(x))} \le \frac{\alpha_W(\delta_{M})}{\alpha_{\mathrm{lo}}^T(\kappa)}.
    $ %
    But Assumption~\ref{asm:superlinear_scaling} implies $\lambda(N) \to \infty$ as $N\to\infty$.
\end{proof}

We have shown uniform convergence of the artificial reference on each fixed set of feasible states, but the closed-loop solution might behave differently for different horizons and leave this set.
Since the offset cost is scaled superlinearly (Assumption~\ref{asm:superlinear_scaling}), we can interpret it as part of the stage cost independently of the horizon.
The lower bounds from Assumptions~\ref{asm:offset_cost_indication} and~\ref{asm:stage_cost_lower_and_upper_bound} imply strict dissipativity with zero storage.
This implies a turnpike property (shown in the next lemma), which yields another feasibility set that is invariant for the closed loop for any horizon (see Lemma~\ref{lem:switching_limits} below).
\begin{lemma}\label{lem:turnpike}
    Suppose Assumptions~\ref{asm:offset_cost_indication}--\ref{asm:superlinear_scaling} hold.
    Then, for all $\Gamma > 0$, there exists $\sigma_{\Gamma} \in \mathcal{L}$ such that for all $N, P \in \mathbb{N}$, $x \in X$, $u \in \mathbb{U}_X^N(x)$ and $r\in\mathcal{R}$ with
    $J^{\mathrm{tc}}_N(x,u,r) + \lambda(N)T(r) \le \Gamma$, the set 
    $\mathcal{Q} = \{ k \in \mathbb{I}_{0:N-1} \mid \vert x_u(k,x) \vert_{x_{\mathrm{d}}} \ge \sigma_{\Gamma} (P) \}$
    has at most $P$ elements.
\end{lemma}
\begin{proof}
    Assumptions~\ref{asm:offset_cost_indication} and \ref{asm:stage_cost_lower_and_upper_bound} imply that $\ell(x,u,r) + T(r) \ge c_1^{\ell} \vert x \vert_{x_r}^2 + \alpha_{\mathrm{lo}}^T(\vert x_r \vert_{x_{\mathrm{d}}})$. Thus, there exists $\rho \in \mathcal{K}_{\infty}$ such that $\ell(x,u,r) + T(r) \ge \rho(\vert x \vert_{x_{\mathrm{d}}})$ for any $x,u,r$.
    From this point on, the proof follows along the lines of~\cite[Prop.~8.15]{Gruene2017}.
    We fix $\Gamma > 0$ and choose $\sigma_{\Gamma}(P) = \rho^{-1}(\frac{\Gamma}{P})$.
    Suppose there exist $N,P,x,u,r$ such that $J^{\mathrm{tc}}_N(x,u,r) + \lambda(N)T(r) \le \Gamma$ but $\mathcal{Q}$ has at least $P+1$ elements.
    However, since $\lambda(N) \ge N$ and $V^\mathrm{f} \ge 0$, we get a contradiction:
    $
        \Gamma \ge J^{\mathrm{tc}}_N(x,u,r) + \lambda(N)T(r)
        \ge \sum_{k=0}^{N-1} \ell(x_u(k,x), u(k), r) + NT(r) \ge \sum_{k=0}^{N-1} \rho(\vert x_u(k,x) \vert_{x_{\mathrm{d}}})
        \ge \sum_{k\in\mathcal{Q}} \rho(\sigma_{\Gamma}(P)) \ge (P+1)\frac{\Gamma}{P} > \Gamma.
    $
\end{proof}

Finally, we are able to show that the error terms in the established performance bound~\eqref{eq:transient_performance_bound} vanish if $N,K\to\infty$.
\begin{lemma}\label{lem:switching_limits}
    Let Assumptions~\ref{asm:offset_cost_indication}--\ref{asm:better_candidate_reference} hold and $\widetilde{N} \in \mathbb{N}$.
    Then, for all $x \in \mathcal{X}_{\widetilde{N}}$, $i \in \{ 1, 2 \}$, with $r_N^*(k) = r_N^*(x_{\mu_N}(k, x))$,
    \begin{equation}
        \lim_{N\to\infty} {\sum_{k=0}^{\infty}} \vert r_N^*(k) \vert_{{r}_{\mathrm{d}}}^i {=} {\sum_{k=0}^{\infty}} \left(\lim_{N\to\infty} \vert r_N^*(k) \vert_{{r}_{\mathrm{d}}}^i \right) = 0.
    \end{equation}
\end{lemma}
\begin{proof}
    First, we show that there exists $P \ge \widetilde{N}$ and $\bar{N}$ such that for all $x \in \mathcal{X}_{\widetilde{N}}$, $N \ge \bar{N}$ and $k\in\mathbb{N}_0$, we have $x_{\mu_N}(k, x) \in \mathcal{X}_{P}$.
    Let $x\in\mathcal{X}_{\widetilde{N}}$.
    From Lemma~\ref{lem:standard_MPC_upper_bounds_value_function}, there exists $\bar{N} \in \mathbb{N}$ such that $V_N(x) \le W_{\bar{N}}(x) \le \alpha_W(\delta_{\widetilde{N}})$ for all $N \ge \bar{N}$, where $\delta_{\widetilde{N}} = \sup_{x \in \mathcal{X}_{\widetilde{N}}} \vert x \vert_{x_{\mathrm{d}}}$.
    We now use Lemma~\ref{lem:turnpike} with $\Gamma = \alpha_W(\delta_{\widetilde{N}})$ and choose $P$ such that $\sigma_{\Gamma}(P) \le c_{\mathrm{b}}$ with $c_{\mathrm{b}}$ from Assumption~\ref{asm:terminal_ingredients}.
    Thus, the set $\{ k \in \mathbb{I}_{0:N-1} \mid \vert x_{u_N^*}(k,x) \vert_{x_{\mathrm{d}}} \ge c_{\mathrm{b}} \}$ has at most $P$ elements for all $x\in\mathcal{X}_{\widetilde{N}}$ and $N \ge \bar{N}$, and, in particular, $x_{u_N^*}(1,x) = x_{\mu_N}(1,x) \in \mathcal{X}_P$.
    In addition, from~\eqref{eq:value_function_decrease_recast}, for all $N \ge \bar{N}$, $V_N(x_{\mu_N}(k,x)) \le V_N(x) \le \alpha_W(\delta_{\widetilde{N}})$ for all $k \in \mathbb{N}_{0}$.
    Hence, Lemma~\ref{lem:turnpike} applies for all $x_{\mu_N}(k,x)$, $k \in \mathbb{I}_{0:N}$, with the same $\Gamma$ and $P$ as before.
    Thus, for all $N \ge \bar{N}$ and $k \in \mathbb{N}_0$, we have $x_{\mu_N}(k,x) \in \mathcal{X}_{P}$.
    Finally, since from Lemma~\ref{lem:convergence_artificial_reference}, $\lim_{N\to\infty} \vert r_N^*(x) \vert_{r_{\mathrm{d}}} = 0$ uniformly on $\mathcal{X}_P$, and since $x_{\mu_N}(k,x) \in \mathcal{X}_P$ for $k \in \mathbb{I}_{0:K}$, together with Lemma~\ref{lem:convergence_of_series}, the claim follows.
\end{proof}
By combining these results, we are able to state our second main result showing that MPC for tracking with a scaled offset cost recovers the optimal infinite horizon cost.
\begin{theorem}\label{thm:asymptotic_performance_bound}
    Suppose Assumptions~\ref{asm:offset_cost_indication}--\ref{asm:better_candidate_reference} hold.
    Then, for any $x \in \mathcal{X}_{\widetilde{N}}$ with $\widetilde{N}\in\mathbb{N}$,
    \begin{equation*}
        \lim_{\substack{N\to\infty \\ K\to\infty}} J_K^{\mathrm{d}}(x, \mu_N) \le \inf_{u\in\mathbb{U}^{\infty}_{\{x_{{\mathrm{d}}}\}}(x)} J_{\infty}^{\mathrm{d}}(x,u).
    \end{equation*}
\end{theorem}
\begin{proof}
    The claim follows from letting $K\to\infty$ and $N\to\infty$ in~\eqref{eq:transient_performance_bound} and applying Lemma~\ref{lem:switching_limits}. 
\end{proof}

\section{Example: Continuous stirred-tank reactor}
We consider the example of a continuous stirred-tank reactor from~\cite{Mayne2011, Koehler2020b}.
The nonlinear system is discretised with an Euler discretisation using a sampling time of $0.1\,\mathrm{s}$.
The goal is to steer the plant from $x_0 = [0.9492, 0.43]$ to the desired reference $x_{\mathrm{e}} = x_{\mathrm{d}} = [0.2632, 0.6519]$.
The constraints are $\mathcal{Z} = ([0, 1] \times [0, 1]) \times [0, 2]$, and we allow equilibria with $\mathcal{Z}_r = ([0.0529, 0.9492] \times [0.43, 0.86]) \times [0.1366, 0.7687]$.
We compute terminal ingredients as proposed in~\cite{Koehler2020a} by gridding $\mathcal{Z}_r$ with 200 points.
We choose a quadratic stage cost with $Q = I_2$ and $R = 0.01$, and $T(r) = 0.01\vert x_{r,1} \vert_{x_{\mathrm{d},1}}^2 + 1000\vert x_{r,2} \vert_{x_{\mathrm{d},2}}^2 + \vert u_r \vert_{u_{\mathrm{d}}}^{2}$.
The simulation is implemented in Python using~\cite{Andersson2019,Diamond2016,Waechter2005}.
In Figure~\ref{fig:state_space_evolution}, we plot the closed-loop solution for different horizons.
It can be observed that a larger horizon allows a more optimal trajectory further from the steady state manifold, since the MPC for tracking has more steps to reach its neighbourhood, and the infinite horizon solution is approximated.
For comparison, we also simulate the closed loop when using standard MPC \eqref{eq:standard_MPC} with $N=515$ and MPC for tracking with $\lambda(N) \equiv 1$.
Note that a standard MPC with $N \le 500$ is infeasible since it cannot satisfy $x_{u_{\mathrm{std}}^*}(N,x) \in \mathcal{X}^{\mathrm{f}}(r_{\mathrm{d}})$.
In this example, MPC for tracking can use much shorter prediction horizons, i.e. $N=1$, which determine the computational burden of MPC.
In addition, it already shows good performance for $N=100$.
This indicates that there is no major disadvantage in terms of performance when using MPC for tracking instead of a standard MPC formulation.
\begin{figure}[bp]
    \setlength\axisheight{0.75\linewidth}
    \setlength\axiswidth{0.95\linewidth}
    \centering
    \input{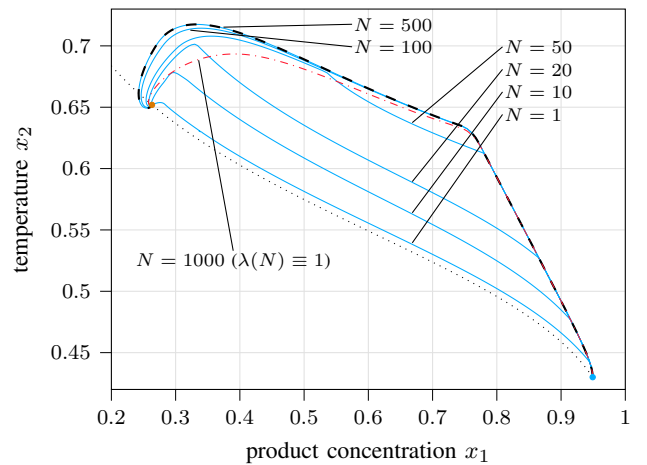}
    \caption{Closed-loop evolution of MPC for tracking with different horizons. The infinite horizon solution is dashed; the solution using standard MPC ($N=515$) is indistinguishable from it. The steady state manifold is dotted.}
    \label{fig:state_space_evolution}
\end{figure}%

\bibliographystyle{IEEEtran}
\bibliography{references}

\end{document}